\documentclass[a4paper, 12pt]{article}

\usepackage{amscd,amsmath,amsthm,amsfonts}
\usepackage{amssymb,color,amsbsy}
\usepackage{fullpage}
\usepackage{amsthm}

\usepackage{tikz}
\usetikzlibrary{arrows,decorations.pathmorphing,backgrounds,positioning,fit}
\usetikzlibrary{petri}
\usepackage{graphicx}
\usepackage{tikz-cd}
\usepackage{tkz-graph}

\usepackage{caption}
\usepackage{subcaption}

\usepackage{multicol}

\usepackage{enumerate}
\usepackage{amsfonts}
\usepackage{color}

\usepackage[pagebackref=true, bookmarksopen=true,colorlinks=true, linkcolor=red,citecolor=blue]{hyperref}

\usepackage{marginnote}

\newtheorem{construction}{Construction}[section]
\newtheorem{theorem}[construction]{Theorem}

\newtheorem{corollary} [construction]{Corollary}

\newtheorem{definition} [construction]{Definition}

\newtheorem{lemma}[construction] {Lemma}

\DeclareMathAlphabet{\mathpzc}{OT1}{pzc}{m}{it}
\newcommand{\Inertia}[1]{\operatorname{Inertia}(#1)}
\newcommand{\Rank}[1]{\operatorname{Rank}(#1)}
\newcommand{\rank}[1]{\operatorname{rank}(#1)}
\newcommand{\Null}[1]{\operatorname{Null}(#1)}
\newcommand{\mnull}[1]{\operatorname{null}(#1)}
\newcommand{\Supp}[1]{\operatorname{Supp}(#1)}
\newcommand{\Core}[1]{\operatorname{Core}(#1)}
\newcommand{\supp}[1]{\operatorname{supp}(#1)}
\newcommand{\core}[1]{\operatorname{core}(#1)}
\newcommand{\npart}[1]{\operatorname{npart}(#1)}
\newcommand{\Npart}[1]{\operatorname{Npart}(#1)}
\newcommand{\down}[3]{{#1}\!\downharpoonleft_{\scalebox{0.5}{#3}}^{\scalebox{0.5}{#2}}}
\newcommand{\up}[3]{{#1}\!\upharpoonleft_{\scalebox{0.5}{#3}}^{\scalebox{0.5}{#2}}}

\begin{document}
\title{On the null structure of bipartite graphs without cycles of length a multiple of $4$}
\author{Daniel A. Jaume\footnote{daniel.jaume.tag@gmail.com}\phantom{i},
Gonzalo Molina\footnote{gonzalo.molina.tag@gmail.com}\phantom{i}, and
Adri\'{a}n Pastine\footnote{adrian.pastine.tag@gmail.com}\phantom{i}\thanks{corresponding author}\\
Universidad Nacional de San Luis\\
Departamento de Matem\'{a}tica\\
San Luis, Argentina}

\maketitle

\begin{abstract}
In this work we study the null space of bipartite graphs without cycles of length multiple of $4$, and its relation to structural properties.
We decompose them into two subgraphs: $C_N(G)$ and $C_S(G)$. 
$C_N(G)$ has perfect matching and its adjacency matrix is nonsingular.
$C_S(G)$ has a unique maximum independent set and the dimension of its null space equals
the dimension of the null space of $G$. Even more, we show that the fundamental spaces of $G$
are the direct sum of the fundamental spaces of $C_N(G)$ and $C_S(G)$.  We also obtain formulas relating the independence number 
and the matching number of a $C_{4k}$-free bipartite graph with $C_N(G)$ and $C_S(G)$, and the dimensions
of the fundamental spaces. Among other results, 
we show that the rank of a $C_{4k}$-free bipartite graph is twice its matching number, generalizing a result for trees due to 
Bevis et al \cite{bevis1995ranks}, and Cvetkovi\'{c} and Gutman \cite{D1972}. About maximum independent sets, we show that the intersection of all maximum independent sets of a
$C_{4k}$-free bipartite graph coincides with the support of its null space. \end{abstract}



\section{Introduction}

All graphs in this work are labeled (even when we do not write the labels), finite, undirected and with neither loops nor multiple edges. Given a graph $G$, the set of vertices is denoted by $V(G)$ and the set of edges by \(E(G)\). Following the style of Bapat, we use uppercase letters for sets and lowercase for their cardinalities (or dimension), see \cite{bapat2014graphs}.  Thus \(e(G)\) is the cardinality of \(E(G)\). 
For all graph-theoretic notions not defined in this article, the reader is referred to~\cite{diestel2000graph}.

Collatz and Sinogowitz (1957), see \cite{von1957spektren}, first raised the problem of characterizing all singular or nonsingular graphs. 
Since then the question has been answered for a few families of graphs. As an example, trees are nonsingular if and only if they have a perfect matching. 

In particular, studying the null space of a family of graph answers the question posed by Collats and Sinogowits about that family. On the other hand, the nullity of a graph has strong chemical implications. The chemical instability of a molecule corresponding to a graph
is expressed in terms of the nullity of the graph. Because of this, many chemists and mathematicians
have studied the nullity of graphs in general. The usual problems studied about the
nullity of a family of graphs includes computing it, finding its distribution, finding
bounds for it, characterizing graphs with certain nullity, and so on.

In this work we study in depth the null space of the adjacency matrix of a particular family of graphs.
Let $G$ be a graph, by $A(G)$  we denote is adjacency matrix. By $\Null{G}$ we denote 
the null space of $A(G)$, and we incur in some notation abuse by calling it the null space of $G$. As stated
before we will use $\mnull{G}$ to denote the cardinality of $\Null{G}$. It is usual to think of \(\Null{G}\) as a subspace of \(\mathbb{R}^{G}\), where \(\mathbb{R}^{G}\)
denotes the vector space of all functions from $V(G)$ to $\mathbb{R}$. In \cite{jaume2018null}, three families of vertices related to the null space of a tree were studied.
The \textit{support} of $G$, $\Supp{G}$, which is the union of the supports (nonzero coordinates) of all vectors in $\Null{G}$;
the \textit{core} of $G$, $\Core{G}$, which is the set vertices that are adjacent to some vertex in $\Supp{G}$;
and the \textit{npart} of $G$, $\Npart{G}$, which is the set of vertices that are not in $\Supp{G}$ nor in $\Core{G}$.
Given a set of vertices $S$, the \textit{neighborhood} of $S$ is the set of vertices $N_G(S)=\{v\in V(G)\,|\,\{v,s\}\in E(G)$ for some $s\in S\}$. When $G$ is clear from the context we write $N(S)$. The neighborhood of $S$ is sometimes called the set of neighbors of $S$. Notice
that $\Core{G}=N(\Supp{G})$, i.e. the core of $G$ is the set of neighbors of the support of $G$. In Figure \ref{fig_1}, vertices in $\Npart{G_1}$ are represented by a square, vertices in $\Supp{G_1}$
by a white circle, and vertices in $\Core{G_1}$ by a black circle.
Thus, we have 
\begin{enumerate}
	\item \(\Supp{G_{1}}=\{2,3,7,8,9,10,13,14,15,21,22,23,24,25,26,26,28\}\),
	\item \(\Core{G_{1}}=\{1,4,5,6,11,12,16,17,18,19,20\}\), and
	\item \(\Npart{G_{1}}=\{29,\dots,56\}\).
\end{enumerate}

The partition of the vertex set using the null structure is not the only partition that we study. In \cite{zito1991structure}, Zito introduced a decomposition based on maximum independent sets. A set of vertices is \textit{independent} if no two vertices in the set are neighbors. An independent set $I$ is 
\textit{maximum} if $|I|\geq |I'|$ for every independent set $I'$. 
Zito used the partition in order to show that the greatest number of maximum independent set for a tree of \(n\) vertices is \(2^{\frac{n-3}{2}}\) for odd \(n>1\) and \(1+2^{\frac{n-2}{2}}\) for \(n\) even. The vertices of a graph were partitioned in three classes: the vertices that are in all maximum independent sets, called static-included vertices, the vertices that are in no maximum independent set, called static-excluded vertices, and the vertices in some but not all maximum independent sets, called flexible vertices. 

A \textit{matching} is a set of edges such that vertices belong to at most one edge in the set, if $\{v,w\}$ is an edge of 
a matching then we say that $v$ is \textit{matched} to $w$, and that they are \textit{saturated} by the matching, or $M$-saturated if $M$ is the matching. A vertex that
is not in any edge of the matching is said to be \textit{unsaturated}, or $M$-unsaturated. A matching $M$ is said to be \textit{maximum} 
if $|M|\geq |M'|$ for every matching $M'$. The Gallai-Edmonds decomposition 
partitions the set of vertices in three classes: the vertices that are unsaturated in some maximum matching, denoted by $\mathcal{D}(G)$; the vertices that are neighbors of vertices in $\mathcal{D}(G)$, but are not themselves in $\mathcal{D}(G)$, denoted by $\mathcal{A}(G)$;
and the vertices that are neither in $\mathcal{D}(G)$ nor in $\mathcal{A}(G)$, denoted by $\mathcal{C}(G)$. See Section 3.2, page 93 in \cite{lovasz2009matching}. 

In the case of trees, the three partitions coincide, see \cite{jaume2018null}. To be more precise, if $T$ is a tree, 
then $\Supp{T}$ is the set of static-included vertices and $\mathcal{D}(T)$; $\Core{T}$ is the set of static-excluded vertices
and $\mathcal{A}(T)$; $\Npart{T}$ is the set of flexible vertices and $\mathcal{C}(T)$.

One of the main results of \cite{jaume2018null} states that the subgraph $\mathcal{F}_S$ induced by $\Supp{T}\cup\Core{T}$
has no npart, and that $\Null{T}$ is the null space of $\mathcal{F}_S$ with $0$ in the coordinates corresponding
to $\Npart{T}$. The subgraph $\mathcal{F}_N$, induced by $\Npart{T}$, has no null space.
They also showed that the maximum matching structure and the maximum independence set structure of $T$ depends only
on $\mathcal{F}_S$ and $\mathcal{F}_N$, and not in the way these subgraphs are connected.

In this work we extend the results from \cite{jaume2018null} on trees to $C_{4k}$-free bipartite graphs, i.e. 
graphs which contain no cycles of length multiple of $4$, and whose set of vertices can be partitioned
into two disjoint sets $X$ and $Y$, such that there are no edges between two vertices in the same set, see Figure \ref{fig_1}.


%
\pgfdeclarelayer{background}
\pgfsetlayers{background,main}

\tikzstyle{vertex}=[circle,fill=black!25,minimum size=18pt,inner sep=0pt]

\tikzstyle{selected vertex} = [vertex, fill=red!24]

\tikzstyle{lado1}=[thick]
\tikzstyle{edge} = [draw,line width=1.2,-]
\tikzstyle{weight} = [font=\small]
\tikzstyle{S edge} = [draw,line width=15pt,-,red!12]
\tikzstyle{N edge} = [draw,line width=15pt,-,blue!12]
\tikzstyle{ignored edge} = [draw,line width=5pt,-,black!20]
\tikzstyle{c vertex} = [circle, draw, fill=black, minimum size=14pt,inner sep=0pt]
\tikzstyle{n vertex} = [rectangle, color=black, fill=white, draw, minimum size=13pt, inner sep=0pt]
\tikzstyle{s vertex}=[circle, draw, fill=white, minimum size=14pt,inner sep=0pt]

\begin{figure}[h]
	\centering
	\begin{tikzpicture}[scale=0.7]
	
	\draw[lado1] (0,0) node[s vertex] (3) [label=below:$1$] {$3$}    		
	-- ++(90:2.0cm) node[c vertex] (1A) {}  		
	-- ++(0:2cm) node[s vertex] (2) [label=right:$1$] {$2$}			
	-- ++(180:2cm) node[c vertex] (1B) {\textcolor{white}{$\mathbf{1}$}}		
	-- ++(90:2cm) 	node (29A) {}				
	-- ++(0:2cm) 	node[n vertex] (30)	 {$30$}			
	-- ++(180:2cm) 	node[n vertex] (29B)	{$29$}		
	-- ++(180:2cm) 	node[c vertex] (5) 	{\textcolor{white}{$\mathbf{5}$}}	
	-- ++(210:2cm) 	node[s vertex] (8) [label=below:$1$]	{$8$}	
	-- ++(150:2cm) 	node[c vertex] (4)	{\textcolor{white}{$\mathbf{4}$}}	
	-- ++(90:2cm) 	node[s vertex] (7) 	[label=left:$-1$] {$7$}	
	-- ++(30:2cm) 	node[c vertex] (6A) 	{}		
	-- ++(330:2cm) 	node[s vertex] (9) [label=right:$-1$] 	{$9$}	
	-- ++(150:2cm) 	node[c vertex] (6B) 	{}		
	-- ++(150:2cm) 	node[s vertex] (10)[label=left:$2$] 	{$10$}	
	-- ++(330:2cm) 	node[c vertex] (6C) 	{\textcolor{white}{$\mathbf{6}$}}	
	-- ++(90:2cm) 	node[n vertex] (31) 	{$31$}		
	-- ++(120:2cm) 	node[n vertex] (32) 	{$32$}		
	-- ++(60:2cm) 	node[n vertex] (33) 	{$33$}		
	-- ++(0:2cm) 	node[n vertex] (34A) 	{}			
	-- ++(300:2cm) 	node[n vertex] (35) 	{$35$}		
	-- ++(240:2cm) 	node[n vertex] (36) 	{$36$}		
	-- ++(300:2cm) 	node[n vertex] (37) 	{$37$}		
	-- ++(0:2cm) 	node[n vertex] (38) 	{$38$}		
	-- ++(60:2cm) 	node[n vertex] (39A) 	{}			
	-- ++(120:2cm) 	node[n vertex] (40) 	{$40$}	;	
	
	\draw[lado1] (34A) 	node (34B)  {$34$}			
	-- ++(90:2cm) 	node[c vertex] (11A)	{}		
	-- ++(180:2cm) 	node[s vertex] (13)	{$13$}	
	-- ++(0:2cm) 	node[c vertex] (11B)	{\textcolor{white}{$\mathbf{11}$}}			
	-- ++(0:2cm) 	node[s vertex] (14)	{$14$}	
	-- ++(0:2cm) 	node[c vertex] (12A)	{}	
	-- ++(0:2cm) 	node[s vertex] (15)	{$15$}	
	-- ++(180:2cm) 	node[c vertex] (12B)	{\textcolor{white}{$\mathbf{12}$}}			
	-- ++(90:2cm) 	node[n vertex] (55A)	{}		
	-- ++(0:2cm) 	node[n vertex] (56)	{$56$}		
	-- ++(180:2cm) 	node[n vertex] (55B)	{$55$}	
	-- ++(180:2cm) 	node[n vertex] (54)	{$54$}		
	-- ++(180:2cm) 	node[n vertex] (53)	{$53$};		
	\draw[lado1] (39A) node[n vertex] (39B) {$39$}		
	-- ++(0:2cm)	node[c vertex] (16A) 	{} 		
	-- ++(120:2cm)	node[s vertex] (21)   {$21$}	
	-- ++(300:2cm)	node[c vertex] (16B)  {}		
	-- ++(240:2cm)	node[s vertex] (22)   {$22$}	
	-- ++(60:2cm)	node[c vertex] (16C)  {\textcolor{white}{$\mathbf{16}$}}	
	-- ++(0:2cm)	node[c vertex] (17A)  {} 		
	-- ++(120:2cm)	node[s vertex] (23)   {$23$}	
	-- ++(300:2cm)	node[c vertex] (17B)  {}		
	-- ++(240:2cm)	node[s vertex] (24)   {$24$}	
	-- ++(60:2cm)	node[c vertex] (17C)  {\textcolor{white}{$\mathbf{17}$}}	
	-- ++(60:2cm) 	node[c vertex] (18A) 	{}		
	-- ++(90:2cm) 	node[s vertex] (25)	[label=above:$1$] {$25$}	
	-- ++(270:2cm)	node[c vertex] (18B)	{\textcolor{white}{$\mathbf{18}$}}	
	-- ++(0:2cm) 	node[s vertex]  (26) [label=above:$-1$] 	{$26$}	
	-- ++(300:2cm) 	node[c vertex] (19A) 	{}		
	-- ++(240:2cm) 	node[s vertex]  (27) [label=below:$1$] 	{$27$}	
	-- ++(180:2cm) 	node[c vertex] (20) 	{\textcolor{white}{$\mathbf{20}$}}	
	-- ++(270:2cm) 	node[s vertex] (28) [label=below:$-1$]	{$28$};	
	
	\draw[lado1] (19A) node[c vertex] (19B) {}		
	-- ++(0:2cm) node[n vertex] (41) {$41$}					
	-- ++(0:2cm) node[n vertex] (42) {$42$};				
	
	\draw[lado1] (19A) node[c vertex] (19C) {\textcolor{white}{$\mathbf{19}$}}		
	-- ++(300:2.cm) 	node[n vertex] (43) {$43$}				
	-- ++(240:2cm)	node[n vertex] (44) {$44$}				
	-- ++(270:2cm) 	node[n vertex] (45) {$45$}				
	-- ++(270:2cm) 	node[n vertex] (46) {$46$}				
	-- ++(300:2cm) 	node[n vertex] (47) {$47$}				
	-- ++(0:2cm)	node[n vertex] (48) {$48$}				
	-- ++(60:2cm)	node[n vertex] (49) {$49$}				
	-- ++(90:2cm)	node[n vertex] (50) {$50$}				
	-- ++(90:2cm)	node[n vertex] (51) {$51$}				
	-- ++(120:2cm)	node[n vertex] (52) {$52$};				
	
	
	\draw[lado1] (5) -- (9);
	\draw[lado1] (17C)-- (20);
	\draw[lado1] (31) -- (36);
	\draw[lado1] (35) -- (40);
	\draw[lado1] (43) -- (52);
	
	
	\foreach \ini / \fin in {1A/2,4/8,5/9,6A/10,11A/14,12A/15,16A/22,17A/23,18A/25,19A/26,20/27,29A/30,31/32,33/34A,35/36,37/38,39A/40,41/42,43/44,45/46,47/48,49/50,51/52,53/54,55A/56}
	\path[edge, decorate, decoration={snake},very thick,color=blue]
	(\ini) -- (\fin);
	
	\begin{pgfonlayer}{background}
	\draw[very thick, red, ->, dashed] (10) 
	-- ++(240:0.45cm) node (10W) {}
	-- ++(330:1.75cm) node (6W) {}
	-- ++(330:2cm) node (9W) {}
	-- ++(270:1.45cm) node (5W) {}
	-- ++(210:1.48cm) node (8W) {}
	-- ++(150:1.46cm) node (4W) {}
	-- ++(90:1.53cm)  node (7W) {};
	
	\draw[very thick, red, ->,dashed] (2) 
	-- ++(270:0.43cm) node (2W) {}
	-- ++(180:1.53cm) node (1W) {}
	-- ++(270:1.53cm) node (3W) {};
	
	\draw[very thick, red, ->,dashed] (25) 
	-- ++(180:0.45cm) node (25W) {}
	-- ++(270:2.45cm) node (18W) {}
	-- ++(0:2.15cm) node (26W) {}
	-- ++(300:1.48cm) node (19W) {}
	-- ++(240:1.45cm) node (27W) {}
	-- ++(180:2.2cm) node (20W) {}
	-- ++(270:2.4cm) node (28W){};
 
	\foreach \source / \dest in {1B/2,1B/3}
	\path[S edge] (\source.center) -- (\dest.center);
	
	\foreach \source / \dest in {5/8,8/4,4/7,7/6A,6A/10,6A/9,9/5}
	\path[S edge] (\source.center) -- (\dest.center);
	
	\foreach \source / \dest in {5/8,8/4,4/7,7/6A,6A/10,6A/9,9/5}
	\path[S edge] (\source.center) -- (\dest.center);
	
	\foreach \source / \dest in {13/11A,11A/14,14/12A,12A/15}
	\path[S edge] (\source.center) -- (\dest.center);
	
	\foreach \source / \dest in {21/16A,16A/22,16A/17A,17A/23,17A/24,17A/18A,25/18A,18A/26,26/19A,19A/27,27/20,20/28,20/17A}
	\path[S edge] (\source.center) -- (\dest.center);
	
	\foreach \source / \dest / \fr in {29A/30}
	\path[N edge] (\source.center) -- (\dest.center);
	
	\foreach \source / \dest / \fr in {31/32,32/33,33/34A,34A/35,35/36,36/31,36/37,37/38,38/39A,39A/40,40/35}
	\path[N edge] (\source.center) -- (\dest.center);
	
	\foreach \source / \dest / \fr in {41/42}
	\path[N edge] (\source.center) -- (\dest.center);
	
	\foreach \source / \dest / \fr in {43/44,44/45,45/46,46/47,47/48,48/49,49/50,50/51,51/52,52/43}
	\path[N edge] (\source.center) -- (\dest.center);
	
	\foreach \source / \dest / \fr in {53/54,54/55A,55A/56}
	\path[N edge] (\source.center) -- (\dest.center);
	
	\end{pgfonlayer}
	\end{tikzpicture}
	\caption{\(G_{1}\) is a \(C_{4k}\)-free bipartite graph, \(M_{1}\) is maximum matching in \(G_{1}\). The nonzero coordinates of a vector \(\vec{x}_{1}\in \Null{G_{1}}\) are shown.}
	\label{fig_1}
\end{figure}
Let \(G\) be a graph with vertex set \(V(G)\) and edge set \(E(G)\). Recall that \(\mathbb{R}^{G}\) denotes the vector space of all functions from \(V(G)\) to \(\mathbb{R}\). Let \(\vec{x} \in \mathbb{R}^{G}\) and \(v \in V(G)\). 
For all linear algebra-theoretic notions not defined here, the reader is referred to~\cite{meyer2000matrix}.

The remainder of this work is organized as follows. In Section \ref{sectionsuppmaxmatch} we introduce and study the notion of \((M,\vec{x})\)-alternating walk. This notion mixes the matching structure, via the matching \(M\), and the null structure, via the vector \(\vec{x}\). We prove that a vertex is in the support of a $C_{4k}$-free graph if and only if it $M$-unsaturated for some maximum matching $M$ and that a \(C_{4k}\)-free bipartite graph is nonsingular if and only if it has a perfect matching. In Section \ref{sectiondecomposition} we give the null decomposition for \(C_{4k}\)-free bipartite graph. In Section \ref{sectionmatch} we prove that the null decomposition coincides with the Gallai-Edmonds decomposition, see \cite{lovasz2009matching}, for $C_{4k}$-free bipartite graphs, and we study the relation between the maximum matching structure and the null decomposition further. As a consequence of the null decomposition we give a generalization to \(C_{4k}\)-free bipartite of a very important result of \cite{bevis1995ranks} and \cite{D1972}: the rank of a tree is twice its matching number. This result is important because it is a bridge between linear algebra and structural graph theory. In Section \ref{sectionspaces} we show that the fundamental spaces of the adjacency matrix of $C_{4k}$-free bipartite graphs
are the direct sum of the fundamental spaces of the subgraphs obtained with the null decomposition. We also extend to $C_{4k}$-free bipartite graphs a result of Gutman, stating that 
the number of positive eigenvalues of a tree coincides with its matching number (this a consequence of the famous Theorem of Sachs about coefficients of the characteristic polynomial of a graph).  In Section \ref{sectionindep} we study the relation between the null decomposition and the structure of maximum independent sets on \(C_{4k}\)-free bipartite graphs. We prove that the null decomposition coincides with the decomposition introduced by Zito in these graphs. Finally, in the brief Section \ref{sectionconclusion} we give some concluding results.

%

%

\section{Support and maximum matchings}\label{sectionsuppmaxmatch}
In this section we study the relation between the support and the structure of maximum matchings of $C_{4k}$-free bipartite graphs. The set of all the maximum matchings of a graph 
\(G\) is denoted by \(\mathcal{M}(G)\). As an example, the set of 
edges marked with a squiggly line in
Figure \ref{fig_1} is one of the 2880 maximum matchings of \(G_1\), 
\(M_{1}\in\mathcal{M}(G_1)\). 

Notice that given a vector $\vec{x}\in \Null{G}$ and a vertex $v$, the sum of $\vec{x}_w$ over all $w\in N(v)$ must be $0$, i.e., $\sum_{w\sim v}\vec{x}_w=0$.
Let $\vec{x}\in \Null{G}$, $M$ a maximum matching of $G$, and  $v\in V(G)$ such that $\vec{x}_v> 0$. If $\{v,u\}\in M$, then $u$ must have a neighbor $w$ with $\vec{x}_w<0$ because its neighbors
add to $0$. Now if $\{w,t\}\in M$, then $t$ must have a neighbor $z$ with $\vec{x}_z>0$, if the graph
is $C_{4k}$-free, then $z\neq v$. One can continue this process until an $M$-unsaturated vertex is reached. 
We will now introduce some notation in order to formalize and properly prove this idea.

\begin{definition}
Let $G$ be a graph, $M$ a matching of $G$, $\vec{x}$ a vector in $\Null{G}$, and 
$v$ a vertex with $\vec{x}_v\neq 0$. Let $W$ be the walk
\[
v=w_0,w_1,\ldots,w_j.
\]
We say that $W$ is an \textit{$(M,\vec{x})$-alternating walk} of $v$ if $\{w_{2i},w_{2i+1}\}\in M$ and $\vec{x}_{w_{2i}} \vec{x}_{w_{2i+2}} <0$, for 
$0\leq i<\lfloor (j-1)/2\rfloor$. 
The \textit{length} of the walk $W$ is $j$.
\end{definition}

The condition $\vec{x}_{w_{2i}} \vec{x}_{w_{2i+2}} <0$ is there just to ensure that the signs of $\vec{x}$ 
in the even positions of the walk are alternating.
It is important to remark that in the preceding $\vec{x}_{w_{2i+1}}$ could have any value, although in $C_{4k}$-free
bipartite graphs these values turn out to be $0$.

With \(vW\) we denote a walk that starts at \(v\), and with \(vWu\) we denote a walk that starts at \(v\) and ends in \(u\). Then \(9W7=9,5,8,4,7\) is a \((M_{1}, \vec{x}_{1})\)-alternating walk from \(9\) to \(7\) in the graph \(G_{1}\) in Figure \ref{fig_1}, where \(\vec{x}_{1}\) is the vector whose coordinate are all zero, but \((\vec{x}_1)_{2}=(\vec{x}_1)_{8}=(\vec{x}_1)_{25}=(\vec{x}_1)_{27}=1\), \((\vec{x}_1)_{3}=(\vec{x}_1)_{7}=(\vec{x}_1)_{9}=(\vec{x}_1)_{26}=(\vec{x}_1)_{28}=-1\), and \((\vec{x}_1)_{10}=2\). Note that \(\vec{x}_1 \in \Null{G_{1}}\). The walk \(10W34=10,6,31,32,33,34\) is not \((M_{1}, \vec{x}_{1})\)-alternating. 
As per usual with the concept of maximal, a \textit{maximal $(M,\vec{x})$-alternating walk} of $v$ is a walk
that is not properly contained in any other $(M,\vec{x})$-alternating walk of $v$. The next lemma states that there are no maximal $(M,\vec{x})$-alternating walk of \(v\) of odd length.

\begin{lemma} \label{lema2.4}
Let \(M\) be a matching in \(G\), \(\vec{x} \in \Null{G}\) and \(v \in \Supp{G}\) such that \(\vec{x}_{v}\neq 0\). Then, every maximal $(M,\vec{x})$-alternating walk of \(v\) has even length.
\end{lemma}

\begin{proof}
	Let \(vW\!u_{k}\) be an $(M,\vec{x})$-alternating walk from \(v\) to \(u_{k}\) of odd length, i.e. \(k\) is odd. Let \(u_{k-1}\) be vertex preceding \(u_{k}\) in \(vW\!u_{k}\), clearly \(\vec{x}_{u_{k-1}}\neq 0\). As  \(\vec{x} \in \Null{G}\), we have that $\sum_{w\sim u_{k}}\vec{x}_w=0$. Then there is \(u_{k+1} \in N(u_{k})\) such that \(\vec{x}_{u_{k-1}}\vec{x}_{u_{k+1}}<0\). Hence, \(W\) can be extend to \(vW\!u_{k+1}\). 
\end{proof}

For example \(2W3=2,1,3\); \(10W7=10,6,9,5,8,4,7\); and \(25W28=25,18,26,19,27,\\20,28\) are maximal \((M_{1}, \vec{x}_{1})\)-alternating walks of \(G_{1}\), see Figure \ref{fig_1}.

The next lemma uses the reasoning of the proof of Lemma \ref{lema2.4}, together with the fact that a $C_{4k}$-free bipartite graph
does not have odd cycles nor cycles of length multiple of $4$, to prove that in such a graph every $(M,\vec{x})$-alternating
walk is a path, and maximal $(M,\vec{x})$-alternating walks end in an $M$-unsaturated vertex.
In order to work with the set of $M$-unsaturated vertices, we denote it by 
$U(M)$. Note that for the graph \(G_{1}\) in Figure \ref{fig_1}, \(U(M_{1})=\{3,7,13,21,24,28\}\).
\begin{lemma}\label{Mxalternatinglemma}
Let $G$ be a $C_{4k}$-free bipartite graph, $M$ a matching of $G$, $\vec{x}$ a vector in $\Null{G}$,
$v$ a vertex with $\vec{x}_v\neq 0$. If $W$ is an $(M,\vec{x})$-alternating walk of $v$,
then $W$ is a path. Furthermore, if $vW\!u$ is a maximal $(M,\vec{x})$-alternating walk from $v$ to $u$,
then \(vW\!u\) is an even path and $u \in U(M)\cap\Supp{G}$.
\end{lemma} 
\begin{proof}
Let $W=w_0,\ldots,w_j$ be an $(M,\vec{x})$-alternating walk of $v$ ($w_0=v$). 
We start by showing that $W$ is a path, in other words that $w_k=w_h$ if and only if $k=h$, for \(k,h \in \{0,\dots,j\}\).

As $G$ is a bipartite graph $w_{2a}\neq w_{2b+1}$ for every $0\leq a,b<\lfloor (j-1)/2\rfloor$.
If $w_{2a+1}=w_{2b+1}$, then $w_{2a}=w_{2b}$ because $\{w_{2i},w_{2i+1}\}\in M$.
Hence, we only need to consider when $w_{2a}=w_{2b}$. There are two cases, $a\equiv b \pmod 2$ and $a\equiv b+1 \pmod 2$.

As $G$ is $C_{4k}$-free, $w_{2a}\neq w_{2b}$ if $a\equiv b \pmod 2$. As $\vec{x}_{w_{2i+2}}\vec{x}_{w_{2i}}<0$, the inequality
$\vec{x}_{w_{2a}}\vec{x}_{w_{2b}}>0$ holds if and only if $a\equiv b \pmod 2$. Hence, $w_{2a}\neq w_{2b}$ if $a\equiv b+1 \pmod 2$. Therefore, $w_k=w_h$ if an only if $k=h$, and $W$ is a path.

Assume $W$ maximal. We must show that the other end of $W$ is in $U(M)\cap\Supp{G}$. By Lemma \ref{lema2.4} the length of $W$ is even.
If $j=2a$, then $w_j$ is in $\Supp{G}$ because $\vec{x}_{w_{2j}}\vec{x}_{w_{2j-2}}<0$ 
and is $M$-unsaturated, as otherwise we could find a longer walk by adding the neighbor through $M$ of $w_j$. Hence, the result follows.
\end{proof}

As can be seen between the lines of the previous definition and result, it is important to see
which vertices can be reached from $\Supp{G}$ through alternating paths. Thus we introduce the following definition.
\begin{definition}
Let $G$ be a graph and $M$ be a maximum matching of $G$. If $X\subset V(G)$ is a set of vertices of $G$,
then the \textit{$M$-even-reachable set} of $X$, $R_e(X,M)$, is the set of vertices that can be reached from
$X$ through an $M$-alternating path of even length. Similarly, the \textit{$M$-odd-reachable set} of $X$, 
$R_o(X,M)$, is the set of vertices that can be reached from $X$ through an $M$-alternating 
path of odd length.
\end{definition}

We usually write \(R_{o}(v,M)\) and \(R_{e}(v,M)\) instead of \(R_{o}(\{v\},M)\) and \(R_{e}(\{v\},M)\) 
when $X$ consists of only one vertex.

Notice that if $G$ is bipartite, then $R_{o}(v,M)\cap R_{e}(v,M)=\emptyset$, because otherwise $G$ would have an odd cycle.

Given a matching $M$, an alternating path $W$ that starts and ends in vertices not saturated by the matching is called an \textit{$M$-augmenting} path. This is due to the fact that the symmetric difference between $M$ and $W$, i.e.
$M\triangle W$ is a matching of greater cardinality. When $M$ is clear from the context,
we say augmenting path instead of $M$-augmenting. A well known result
of the theory of matchings states that a matching is maximum if and only if 
it does not have an augmenting path, see ~\cite{diestel2000graph}.

Assume that $v$ is $M$-unsaturated for some maximum matching $M$ of a \(C_{4k}\)-free graph. Notice that every vertex in 
$R_o(v,M)$ must be $M$-saturated, otherwise the path joining $v$ to the $M$-unsaturated vertex would be an $M$-augmenting path, which contradicts the fact that $M$ is a maximum 
matching. Furthermore, if $w\in R_e(v,M)$, then $N(w)\subset R_o(v,M)$, as the even alternating path from $v$ to $w$ must finish with an edge in $M$, and every other neighbor of $w$ must be connected to it with an edge not in $M$. Notice also that the subgraph induced by  $R_o(v,M)\cup R_e(v,M)- v$ has a perfect matching, thus $|R_e(v,M)|=|R_o(v,M)|+1$. 
Furthermore, as $G$ is bipartite, $R_o(v,M)$ and $R_e(v,M)$ are independent sets because if $u,w\in V(G)$ are neighbors, and $u,w\in R_o(v,M)$ or $u,w\in R_e(v,M)$),  then the closed walk that goes from $u$ to $v$ to $w$ to $u$ is an odd closed walk. 
We summarize these observations in the next lemma.

\begin{lemma} \label{Lemma_Nuevo1}
	Let \(G\) be a \(C_{4k}\)-free bipartite graph, and \(M \in \mathcal{M}(G)\). If \(v \in U(M)\), then 
	\begin{enumerate}
		\item \(R_{e}(v,M) \cap R_{o}(v,M)=\emptyset\),
		\item \(R_{o}(v,M) \cap U(M)=\emptyset\),
		\item \(R_{e}(v,M)|=|R_{o}(v,M)|+1\),
		\item \(N(R_{e}(v,M)) \subset R_{o}(v,M)\),
		\item \(R_{e}(v,M)\) and \(R_{o}(v,M)\) are independent sets of \(G\).
	\end{enumerate}
\end{lemma}

In the proof of Lemma \ref{Mxalternatinglemma}, $u$ is not saturated by $M$, and $v\in R_{e}(u,M)$.
Rewriting said lemma in terms of $U(M)$ yields the following.
\begin{corollary}\label{lemmasuppalternatingpath}
Let $G$ be a $C_{4k}$-free bipartite graph and $M$ a maximum matching of $G$. 
If $v\in \Supp{G}$, then $v\in R_e(U(M),M)$.
\end{corollary}
\begin{proof}
Let $\vec{x}\in\Null{G}$, with $\vec{x}_v\neq 0$.
Let $W$ be a maximal $(M,x)$-alternating walk of $v$. Then by Lemma \ref{Mxalternatinglemma} the other end of $W$ is
an $M$-unsaturated vertex, and the length of $W$ is even. Therefore $v\in R_e(U(M),M)$.
\end{proof}

Using now that $\Core{G}=N(\Supp{G})$, we obtain $\Core{G}\subset R_o(U(M),M)$.
\begin{corollary}\label{coreinroumm}
Let $G$ be a $C_{4k}$-free bipartite graph and $M$ a maximum matching of $G$. 
If $v\in \Core{G}$, then $v\in R_o(U(M),M)$.
\end{corollary}

We use the notion of augmenting path to prove that
$\Supp{T}\cap \Core{T}=\emptyset$. Thanks to this result we can actually talk about 
$\Supp{G}, \Core{G}$ and $\Npart{G}$ as a partition of the vertex set.
This property does not hold for bipartite graphs in general. For instance $\Supp{C_4}=\Core{C_4}=V(C_4)$.
The property does hold for $C_{4k}$-free bipartite graphs, as the following result states.

\begin{corollary}\label{corolariosuppcorenointer}
If $G$ is a $C_{4k}$-free bipartite graph, then $\Supp{G}\cap \Core{G}=\emptyset$.
\end{corollary}
\begin{proof}
Suppose $x\in\Supp{G}\cap \Core{G}$. Then $x\in R_e(y,M)$ for some vertex $y\in U(M)$, and $x\in R_o(z,M)$ for some
vertex $z\in U(M)$. Notice that the alternating walk that goes from $y$ to $x$ and then from $x$ to $z$ starts and 
finishes in unsaturated vertices. An alternating path can be obtained from this walk by deleting
the vertices between repetitions of a vertex. This yields an augmenting path. This contradicts the fact that $M$ 
is a maximum matching. Therefore, $\Supp{G}\cap \Core{G}=\emptyset$.
\end{proof}

As $\Core{G}$ is the set of neighbors of vertices in $\Supp{G}$, the previous result can be restated
to say that $\Supp{G}$ is an independent set.
\begin{corollary} \label{L_ind_Supp}
	If \(G\) is $C_{4k}$-free bipartite graph, then \(\Supp{G}\) is an independent set of \(G\).
\end{corollary}
\begin{proof}
The corollary follows from Corollary \ref{corolariosuppcorenointer}, by replacing $\Core{G}$ by the set of neighbors
of $\Supp{G}$.
\end{proof}

The matching number of a graph $G$, $\nu(G)$, is the size of a maximum matching, i.e. if $M$
is a maximum matching of $G$, then $\nu(G)=|M|$.
If $\nu(G)=V(G)/2$, then every maximum matching of $G$ is also called a \textit{perfect matching},
and the graph $G$ is said to have  perfect matching. A known result states
that a tree has perfect matching if and only if its adjacency matrix is nonsingular, see \cite{bapat2014graphs}.
We extend said result to the case of $C_{4k}$-free bipartite graphs.
\begin{corollary}\label{perfmatchnonsin}
Le $G$ be a $C_{4k}$-free bipartite graph. If $G$ has a perfect matching, then $A(G)$ is a nonsingular matrix.
\end{corollary}
\begin{proof}
As a maximum matching of $G$ is perfect, it saturates all vertices of $G$. By Corollary \ref{lemmasuppalternatingpath},
$\Supp{G}=\emptyset$. Therefore $\Null{G}=\{\vec{0}\}$ and $A(G)$ is nonsingular.
\end{proof}

The reciprocal result is a direct consequence of Theorem \ref{maintheorem}.
 
Given a maximum matching $M$ and an $M$-unsaturated vertex $v$, a new maximum matching can be obtained by exchanging
the edges in $M$ with the edges not in $M$ in an $M$-alternating path of even length connecting $v$ to another
vertex. This fact can be used in conjunction with Corollary \ref{lemmasuppalternatingpath} to prove that if a vertex
$v$ is in the support of a $C_{4k}$-free bipartite graph, then $v$ is $M$-unsaturated for some maximum
matching $M$. The reciprocal can be proved by taking the subgraph induced by $M$-saturated vertices.
As it has a perfect matching, Corollary \ref{perfmatchnonsin} can be applied.
This characterization is the main result of this section.
\begin{theorem}\label{maintheorem}
Let $G$ be a $C_{4k}$-free bipartite graph, and $v\in V(G)$. Then $v\in \Supp{G}$ if and only if $v$ is 
$M$-unsaturated for some maximum matching $M$.
\end{theorem}

We split the proof of Theorem \ref{maintheorem} into two lemmas, one for each implication, because the proofs are
quite different.
\begin{lemma}\label{mainida}
Let $G$ be a $C_{4k}$-free bipartite graph, and $v\in \Supp{G}$. Then $v$ is 
$M$-unsaturated for some maximum matching $M$.
\end{lemma}
\begin{proof}
Assume $v\in \Supp{G}$ and let $M'$ be a maximum matching of $G$.
By Corollary \ref{lemmasuppalternatingpath} there is an $M'$-alternating path of even length from $v$ 
to an $M'$-unsaturated vertex $w$, let $P$ be the set of edges in the path. Then $M=M'\triangle P$, the symmetric difference between $M'$ and $P$, is a maximum matching that does not saturate $v$. Therefore $v$ 
is $M$-unsaturated for some maximum matching $M$.
\end{proof}

Let $F\subset V(G)$. The subgraph \textit{induced} by $F$, denoted $G\langle F \rangle$, 
is the graph with vertex set $F$
and $E(G\langle F \rangle)=\{e\in E(G),|\, e\subset F\}$.
In order to prove the reciprocal, given $v\in U(M)$ we use the subgraph $H$, induced by 
$R_e(v,M)\cup R_o(v,M)$ and with $H'$, induced by $V(H)-v$, in order to obtain a vector in $\Null{G}$. 
We first show the process in an example. Take vertex $7$ from the graph in Figure \ref{fig_1}, using the 
edges marked with squiggly lines as the maximum matching, $M_1$. Then $V(H)=\{4,5,6,7,8,9,10\}$.
Ordering the vertices as $7,4,5,6,8,9,10$, the adjacency matrix of $H$ is
\[
A(H)=\left[\begin{array}{ccccccc}
0 & 1 & 0 & 1 & 0 & 0 & 0\\
1& 0 & 0 & 0 & 1 & 0 & 0\\
0 & 0 & 0 & 0 & 1 & 1 & 0\\
1 & 0 & 0 & 0 & 0 & 1 & 1\\
0 & 1 & 1 & 0 & 0 & 0 & 0\\
0 & 0 & 1 & 1 & 0 & 0 & 0\\
0 & 0 & 0 & 1 & 0 & 0 & 0\\
\end{array}
\right],
\]
and the adjacency matrix of $H'$ is   
\[
A(H')=\left[\begin{array}{cccccc}
 0 & 0 & 0 & 1 & 0 & 0\\
 0 & 0 & 0 & 1 & 1 & 0\\
 0 & 0 & 0 & 0 & 1 & 1\\
 1 & 1 & 0 & 0 & 0 & 0\\
 0 & 1 & 1 & 0 & 0 & 0\\
 0 & 0 & 1 & 0 & 0 & 0\\
\end{array}
\right].
\]
Notice that 
\[
A(H)\left[\begin{array}{ccccccc}
1\\
0\\
0\\
0\\
0\\
0\\
0\\
\end{array}
\right]=\left[\begin{array}{ccccccc}
0\\
1\\
0\\
0\\
1\\
0\\
0\\
\end{array}
\right],
\]
this is, multiplying the adjacency matrix of $H$ by the vector that has a $1$ in the entry corresponding to $v$,
and $0$ everywhere else, gives the vector that has a $1$ in the entries corresponding to neighbors of $v$ and $0$
everywhere else. Notice in particular that, as $v$ is not its own neighbor,
the coordinate corresponding to $v$ in the product is $0$. Had we used $A(G)$, the adjacency matrix of $G$, instead of $A(H)$, we would
have obtained the same vector, with $0$ in the extra coordinates.

Now, $H'$ has a perfect matching. Hence, by Corollary \ref{perfmatchnonsin}, $A(H')$ is nonsingular. We solve the equation $A(H')\vec{x}=A(H)_1$, where $A(H)_1$ is the first column of $A(H)$, deleting the first entry (i.e., the vector of neighbors of $v$, without the entry corresponding to $v$ itself).
\[
\left[\begin{array}{cccccc}
 0 & 0 & 0 & 1 & 0 & 0\\
 0 & 0 & 0 & 1 & 1 & 0\\
 0 & 0 & 0 & 0 & 1 & 1\\
 1 & 1 & 0 & 0 & 0 & 0\\
 0 & 1 & 1 & 0 & 0 & 0\\
 0 & 0 & 1 & 0 & 0 & 0\\
\end{array}
\right]\vec{x}=\left[\begin{array}{ccccccc}
1\\
0\\
1\\
0\\
0\\
0\\
\end{array}
\right],
\]
the solution obtained is 
\[
\vec{x}=\left[\begin{array}{r}
0\\
0\\
0\\
-1\\
1\\
2\\
\end{array}
\right],
\]
now we construct the vector $\vec{y}\in\Null{H}$ by putting a $-1$ as first coordinate:
\[
\vec{y}=\left[\begin{array}{r}
-1\\
\vec{x}
\end{array}
\right]=
\left[\begin{array}{r}
-1\\
0\\
0\\
0\\
-1\\
1\\
2\\
\end{array}
\right].
\]
Notice that the way in which $\vec{y}$ was constructed ensures that it is in $\Null{H}$.
\[
\left[\begin{array}{ccccccc}
0 & 1 & 0 & 1 & 0 & 0 & 0\\
1& 0 & 0 & 0 & 1 & 0 & 0\\
0 & 0 & 0 & 0 & 1 & 1 & 0\\
1 & 0 & 0 & 0 & 0 & 1 & 1\\
0 & 1 & 1 & 0 & 0 & 0 & 0\\
0 & 0 & 1 & 1 & 0 & 0 & 0\\
0 & 0 & 0 & 1 & 0 & 0 & 0\\
\end{array}
\right] \left[\begin{array}{r}
-1\\
0\\
0\\
0\\
-1\\
1\\
2\\
\end{array}
\right]=\left[\begin{array}{ccccccc}
0\\
0\\
0\\
0\\
0\\
0\\
0\\
\end{array}
\right].
\]
Finally, in order to obtain a vector in $\Null{G}$, we complete $\vec{y}$ by placing $0$ in each coordinate
corresponding to vertices not in $H$.
\begin{lemma}\label{mainvuelta}
Let $G$ be a $C_{4k}$-free bipartite graph, $M$ a maximum matching of $G$, and $v\in U(M)$. Then $v\in \Supp{G}$.
\end{lemma}
\begin{proof}
Let $H$ be the subgraph of $G$ induced by $R_e(v,M)\cup R_o(v,M)$. Then, by Lemma \ref{Lemma_Nuevo1}, $H$ is bipartite, and its bipartition is given 
by $R_e(v,M)$ and $R_o(v,M)$. Let $A(H)$ be the adjacency matrix of $H$, with $v$ as its first row/column. Let $H'$ be the subgraph of $H$ 
obtained by deleting $v$, and let $A(H')$ be its adjacency matrix. Notice that $H'$ has a perfect matching 
(given by the restriction of $M$ to $H'$).  By Corollary \ref{perfmatchnonsin} $A(H')$ 
is nonsingular. Let $\vec{x}\in \mathbb{R}^{H'}$ be the solution to
\[
A(H')\vec{x}=A(H)_1,
\]
where $A(H)_1$ is the first column of $A(H)$ without its first coordinate. Let $\vec{y}\in \mathbb{R}^{H}$ be the extension of 
$\vec{x}$ obtained 
by adding a $-1$ as first coordinate:
\[
\vec{y}=\left[\begin{array}{r}
-1\\
\vec{x}
\end{array}
\right].
\]
Then $\vec{y}\in \Null{A(H)}$. Let $\vec{z}\in \mathbb{R}^{G}$ be the extension of 
$\vec{y}$ obtained by 
filling with $0$ the remaining coordinates. Then $\vec{z}\in \Null{G}$, because
$N(R_e(v,M))=R_o(v,M)$. As $\vec{z}_v=-1$, $v\in \Supp{G}$, as we wanted to prove.
\end{proof}

We can now prove the reciprocal of Corollary \ref{perfmatchnonsin}.
\begin{corollary} \label{Teo_matching_perfect}
	Let $G$ be a $C_{4k}$-free bipartite graph. $G$ has a perfect matching if and only $A(G)$ is a nonsingular matrix.
\end{corollary}
\begin{proof}
	Let \(G\) be a $C_{4k}$-free bipartite such that $A(G)$ is a nonsingular matrix, and let $M$
	be a maximum matching of $G$. By Theorem \ref{maintheorem} $U(M)=\emptyset$. Hence, \(G\) has a perfect matching.
\end{proof}
%
\section{Null decomposition for $C_{4k}$-free bipartite graphs}\label{sectiondecomposition}
In this section we are going to extend the decomposition presented in \cite{jaume2018null}
to $C_{4k}$-free bipartite graphs. First we present a characterization of 
$\Supp{G}$ and $\Core{G}$ in terms of how they can be reached from the vertices
not saturated by a maximum matching.

Notice that $R_e(U(M),M)\subset \Supp{G}$ by Theorem \ref{maintheorem}. On the other hand, 
$\Supp{G}\in R_e(U(M),M)$ by Corollary \ref{lemmasuppalternatingpath}. Furthermore, $\Core{G}=R_o(U(M),M)$,
as $R_o(U(M),M)$ is the set of neighbors of $R_e(U(M),M)$. This yields the following.

\begin{lemma}\label{lemmareachable}
Let $G$ be a $C_{4k}$-free bipartite graph and $M$ a maximum matching of $G$.
Then $R_e(U(M),M)=\Supp{G}$ and $R_o(U(M),M)=\Core{G}$.
\end{lemma}
\begin{proof}
The lemma follows from the discussion preceding it.
\end{proof}

Note that Theorem \ref{maintheorem} says that if \(G\) a \(C_{4k}\)-free graph, then its support, \(\Supp{G}\), is the \(\mathcal{D}(G)\) set of the  Gallai-Edmonds Structure Theorem, see Theorem 3.2.1, page 94 in \cite{lovasz2009matching}. Hence, by Corollary \ref{corolariosuppcorenointer}, \(\Core{G}\) is the \(\mathcal{A}(G)\) set of Gallai-Edmonds, and \(\Npart{G}\) is the \(\mathcal{C}(G)\) set of Gallai-Edmonds. Then, by the Gallai-Edmonds Structure Theorem,
\begin{enumerate}
	\item \(G\!\left\langle \Npart{G}\right\rangle\) has a perfect matching,
	\item if \(M \in \mathcal{M}(G)\), then \(M\) contains a perfect matching of each component of \(G\!\left\langle \Npart{G}\right\rangle\), and matches all the core vertices to supported vertices, \label{C_GE_2}
	\item for all \(U \subset \Core{G}\), \(|U|\leq |N(U) \cap \Supp{G}|\),
	\item \(\nu(G)=\core{G}+\dfrac{\npart{G}}{2}\).
\end{enumerate}
The next two results are consequences of the Gallai-Edmonds Structure Theorem. Here we use that each matching $M$ defines a bijection on \(V(G)\): \(M(v)=v\) if $v\in U(M)$, and \(M(v)=u\) if \(\{vu\} \in M\).

\begin{corollary}\label{corematchedtosupp}
Let $G$ be a $C_{4k}$-free bipartite graph and $M$ a maximum matching of $G$. Then 
$M(\Core{G})\subset \Supp{G}$.
\end{corollary}

\begin{corollary}\label{nmatchedton}
Let $G$ be a $C_{4k}$-free bipartite graph and $M$ a maximum matching of $G$. Then $M(\Npart{G})=\Npart{G}$.
\end{corollary}

We are ready to introduce the subgraphs of the null decomposition. In \cite{jaume2018null}, the authors
introduced the null decomposition of a tree $T$ into two forests, $F_{S}(T)$ and $F_{N}(T)$.
Our subgraphs are obtained in a similar fashion, but are not necessarily forests. Because of this
we decided to use a different letter for the subgraphs.

Given a graph $G$, the $S$-subgraph of $G$ is the subgraph induced by $\Supp{G}\cup \Core{G}$, and is 
denoted by $C_S(G)$, i.e. \(C_S(G)=G\left\langle \Supp{G} \cup \Core{G}\right\rangle\). Similarly, the $N$-subgraph of $G$ is the subgraph induced by $\Npart{G}$, and is 
denoted by $C_N(G)$, i.e. \(C_{N}(G)=G\left\langle \Npart{G} \right\rangle\). In Figure \ref{fig_1} the edges in different subgraphs have by shadows.

The following theorem shows that a $C_{4k}$-free bipartite graph can be decomposed into
$C_S(G)$ and $C_N(G)$, and that the information for the null space (and therefore matching structure 
and independence structure) remains in $C_S(G)$ and $C_N(G)$.

\begin{theorem}\label{descomposicion}
Let $G$ be a $C_{4k}$-free bipartite graph. Then
\begin{enumerate}
	\item $\Supp{C_S(G)}=\Supp{G}$,
	\item $\Core{C_S(G)}=\Core{G}$, and 
	\item $\Npart{C_N(G)}=\Npart{G}$.
\end{enumerate}
Furthermore, $\Npart{C_{S}(G)}=\Supp{C_N(G)}=\Core{C_{N}(G)}=\emptyset$.
\end{theorem}	
%
%
%

\begin{proof}
By Corollary \ref{nmatchedton}, $C_N(G)$ has a perfect matching. Then 
$\Supp{C_N(G)}=\emptyset$, 
$\Core{C_N(G)}=\emptyset$ and $\Npart{C_N(G)}=V(C_N(G))=\Npart{G}$.

Let $M$ be a maximum matching of $G$, and let $M'$ be the restriction of $M$ to $C_S(G)$.
Then $R_e(U(M),M)=R_e(U(M'),M')$ and $R_o(U(M),M)=R_o(U(M'),M')$. Therefore 
$\Supp{C_S(G)}=\Supp{G}$, 
$\Core{C_S(G)}=\Core{G}$, $\Npart{C_S(G)}=\emptyset$ and $\mnull{G}=\mnull{C_S(G)}$.
\end{proof}

Most of the structure of $C_S(G)$ and $C_N(G)$ can be studied from their connected components.
Given a graph $G$, the set of connected components of $G$ is denoted by $\mathcal{K}(G)$. The next
corollary follows from writing Theorem \ref{descomposicion} in terms of $\mathcal{K}(C_S(G))$ and
$\mathcal{K}(C_N(G))$. 
\begin{corollary}
	Let \(G\) be a \(C_{4k}\)-free bipartite graph. Then
	\begin{enumerate}
		\item for all \(H \in \mathcal{K}(C_{S}(G))\), \(\Supp{H}=\Supp{G}\cap V(H)\),
		\item for all \(H \in \mathcal{K}(C_{S}(G))\), \(\Core{H}=\Core{G}\cap V(H)\),
		\item  for all \(H \in \mathcal{K}(C_{N}(G))\), \(\Npart{H}=\Npart{G}\cap V(H)\).
	\end{enumerate}
\end{corollary}

%

\section{Further implications on maximum matchings}\label{sectionmatch}
In this section we study the implications of 
Theorem \ref{maintheorem} and Theorem \ref{descomposicion}
for the maximum matching structure of a $C_{4k}$-free bipartite graph.

Writing Corollary \ref{Teo_matching_perfect} in terms of the null decomposition yields the following characterization
result and it is an advance in the Collatz-Sinogowitz program of  characterizing non-singular graphs. 
\begin{corollary}\label{charactcn}
If $G$ is a $C_{4k}$-free bipartite graph, then the following statements are equivalent.
\begin{enumerate}
\item $G$ has a perfect matching.
\item$A(G)$ is a nonsingular matrix.
\item $G=C_N(G)$.
\end{enumerate}
\end{corollary}

Bevis et al (1995), \cite{bevis1995ranks}, proved that given a tree $T$, $\nu(T)=2\rank{T}$. Earlier, via the famous Sachs Theorem,  Cvetkovi\'{c} and Gutman (1972), see \cite{D1972}, proved that $\mnull{T}=|V(T)|-2\nu(T)$. This is an important
result, because of how elegantly it relates structural and spectral properties of trees.
The following two results are the generalization to $C_{4k}$-free bipartite
graphs.
Notice that $|U(M)|$ equals the amount of vertices in $G$, minus the amount of edges in $M$. In particular, if $M$ is a maximum matching $|U(M)|=|V(G)|-2\nu(G)$.
\begin{theorem}\label{corNullequality}
If $G$ is a $C_{4k}$-free bipartite graph, then 
$\mnull{G}= |V(G)|-2\nu(G)$.
\end{theorem}
\begin{proof}
Let $M$ be a maximum matching of $G$. For each vertex $x$ not saturated by $M$ construct a vector 
$v^{(x)}$ in the null space as in Theorem 
\ref{maintheorem}. This vectors are linearly independent as for each $x$, $v^{(x)}$ is the only 
vector with a nonzero $x$-coordinate.
Hence, $\mnull{G}\geq |U(M)| = |V(G)|-2\nu(G)$.

For the other inequality, notice that by deleting the rows and columns corresponding to vertices 
not saturated by $M$, we obtain the adjacency matrix of a subgraph with perfect matching. By 
Corollary \ref{Teo_matching_perfect} this submatrix is nonsingular. Thus, the original matrix has at 
least $2\nu(G)$ linearly independent columns, and $\mnull{G}\leq |V(G)|-2\nu(G)$.
%
\end{proof}

Rewriting Theorem \ref{corNullequality} in terms of the rank of $G$ we obtain the following:
\begin{corollary}\label{rank2nu}
If $G$ is a $C_{4k}$-free bipartite graph, then $\rank{G}=2\nu(G)$.
\end{corollary}

Rewriting Theorem \ref{corNullequality} in terms of $U(M)$, yields
$\mnull{G}=|U(M)|$.

\begin{corollary} \label{Gordsilresult}
	Let \(G\) be a \(C_{4k}\)-free bipartite graph and \(M\) a maximum matching of $G$. Then \(\mnull{G}=|U(M)|\).
\end{corollary}

By Gallai-Edmonds Structure Theorem every maximum matching of a $C_{4k}$-free bipartite graph $G$
is the union of a maximum matching of $C_S(G)$ and a maximum matching of $C_N(G)$. This result is due to the fact 
that no edge of a maximum matching has a vertex in $\Core{G}$ and another vertex in $\Npart{G}$.
\begin{corollary}\label{matchingdecompo}
Let $G$ be a $C_{4k}$-free bipartite graph.
If $M$ is a maximum matching of $G$, then $M\cap E(C_S(G))$ is a maximum matching of $C_S(G)$
and $M\cap E(C_N(G))$ is a maximum (perfect) matching of $C_N(G)$.
\end{corollary}

We finalize this section with a stability result.
\begin{corollary} \label{coro_establididad}
	Let \(G\) be a \(C_{4k}\)-free bipartite graph and  \(M \in \mathcal{M}(G)\). If \(v \in \Npart{G}\), then 
	\begin{enumerate}
		\item \(\nu(G-v-M(v))=\nu(G)-1\),
		\item \(\mnull{G-v-M(v)}=\mnull{G}\),
		\item \(\rank{G-v-M(v)}= \rank{G}-2\).
	\end{enumerate}
\end{corollary}

\begin{proof}
	As \(\{v,M(v)\} \in M\), it is clear that \(\nu(G-v-M(v))=\nu(G)-1\) (this is true for every graph). Hence, by Theorem \ref{corNullequality}, \(\mnull{G-v-M(v)}=\mnull{G}\), and by Corollary \ref{rank2nu}, \(\rank{G-v-M(v)}= \rank{G}-2\).
\end{proof}

%
\section{On fundamental spaces}\label{sectionspaces}
It is well known that the fundamental spaces (i.e. the rank and the null space) of the adjacency matrix of a graph 
are the direct sum of the fundamental spaces of its connected components.
In this section we prove that for $C_{4k}$-free bipartite graphs, the fundamental spaces of 
the adjacency matrix are the direct sum of the fundamental spaces of $C_S(G)$ and $C_N(G)$.

Using Corollary \ref{matchingdecompo} one can prove that the null decomposition provides a decomposition
of the fundamental spaces of the adjacency matrix. As $M\cap E(C_S(G))$ is a maximum matching of $C_S(G)$,
$\mnull{G}=\mnull{C_{S}(G)}$, which in turns equals \(\sum _{H \in \mathcal{K}(C_{S}(G))} \mnull{H}\),
as the fundamental spaces of any graph can be decomposed in terms of the fundamental spaces of its connected
components.

In order to state the results from this section, we need the following notation introduced in \cite{jaume2018null}. Given a graph \(G\), let \(\vec{x}\) be a vector of \(\mathbb{R}^{G}\), and \(H\) be a subgraph of \(G\). The vector obtained when restricting \(\vec{x}\) to the coordinates (vertices) associated with \(H\) is denoted by \(\down{\vec{x}}{$G$}{$H$}\). By  \(\up{\vec{y}}{G}{H}\) we denote the lift of vector \(\vec{y} \in \mathbb{R}^{H}\) to a vector of \(\mathbb{R}^{G}\): for any \(u \in V(G) - V(H)\), \((\up{\vec{y}}{$G$}{$H$})_{u}:=0\), and for any \(u \in V(H)\), \((\up{\vec{y}}{$G$}{$H$})_{u}:=\vec{y}_{u}\). This notation can be extended naturally to sets of vector: \(\up{Y}{$G$}{$H$}:=\{\up{\vec{x}}{$G$}{$H$}: \vec{x} \in Y\}\) and \(\down{Y}{$G$}{$H$}:=\{\down{\vec{x}}{$G$}{$H$}: \vec{x} \in Y\}\). 
\begin{theorem}\label{decompositionandnullspace}
Let \(G\) be a \(C_{4k}\)-free bipartite graph. Then
\begin{enumerate}
	\item \(\mnull{G}=\mnull{C_{S}(G)}=\sum \limits_{H \in \mathcal{K}(C_{S}(G))} \mnull{H}\),
	\item \(\Null{G}= \up{\Null{C_{S}(G)}}{$G$}{$C_{S}(G)$}\),
	\item \(\Null{C_{S}(G)}=\bigoplus\limits_{H \in \mathcal{K}(C_{S}(G))} \up{\Null{H}}{$C_{S}(G)$}{$H$}\),
	\item \(\Null{G}=\bigoplus\limits_{H \in \mathcal{K}(C_{S}(G))} \up{\Null{H}}{$G$}{$H$}\).
\end{enumerate}
\end{theorem}
\begin{proof}
\begin{enumerate}
\item This result follows from the discussion preceding the theorem.
\item 
	Let $\vec{x}\in \Null{G}$, and let $\pi(\vec{x})$ be the projection of $\vec{x}$ into $\mathbb{R}^{V(C_S(G))}$, i.e.
	$\pi(\vec{x})$ is the vector obtained by deleting the coordinates corresponding to vertices not in $C_S(G)$. 
	We will show that 
	$\pi(\vec{x})\in \Null{C_S(G)}$. Let $v\in V(G))$, then
	\[
	\sum_{w\in N_G(v)}\vec{x}_w=\sum_{w\in \left(N_G(v)\cap \Supp{G}\right)}\vec{x}_w
	=\sum_{w\in N_{C_S(G)}(v)}\vec{x}_w
	=\sum_{w\in N_{C_S(G)}(v)}\pi(\vec{x})_w.
	\]
	Hence, $\pi(\vec{x})\in \Null{C_S(G)}$. Then \(\Null{G}= \up{\Null{C_{S}(G)}}{$G$}{$C_{S}(G)$}\) because $\mnull{G}=\mnull{C_S(G)}$.
\item This result follows from the fact that the fundamental spaces of a graph are the direct sum of the fundamental
spaces of its connected components.
\item This result follows from combining items (ii) and (iii).
\end{enumerate}
\end{proof}

As the adjacency matrix of a graph is a real symmetric matrix, $\Rank{G}$ is the orthogonal complement of $\Null{G}$.
Applying this fact to 
Theorem \ref{decompositionandnullspace} yields the following.
\begin{theorem}\label{Big_theo}
	Let \(G\) be a \(C_{4k}\)-free bipartite graph. Then
\begin{enumerate}
	\item \(\rank{G}=\rank{C_{S}(G)}+\rank{C_{N}(G)}=\sum \limits_{H \in \mathcal{K}(C_{S}(G))} \rank{H}+\sum \limits_{H \in \mathcal{K}(C_{N}(G))} \rank{H}\),
	\item \(\Rank{G}= \up{\Rank{C_{S}(G)}}{$G$}{$C_{S}(G)$} \oplus \;\; \up{\Rank{C_{N}(G)}}{$G$}{$C_{N}(G)$}\),
	\item \(\Rank{G}= \bigoplus\limits_{H \in \mathcal{K}(C_{S}(G))} \up{\Rank{H}}{$G$}{$H$} \oplus\;\;\bigoplus\limits_{H \in \mathcal{K}(C_{N}(G))} \up{\Rank{H}}{$G$}{$H$} \).
\end{enumerate}
\end{theorem}
\begin{proof}
\begin{enumerate}
\item The result is obtained from combining (i) of Theorem \ref{decompositionandnullspace}
with the fact that $|V(G)|=\rank{G}+\mnull{G}$.
\item $\Rank{G}$ is the orthogonal complement of $\Null{G}=\up{\Null{C_{S}(G)}}{$G$}{$C_{S}(G)$}=\Null{C_{S}(G)\cup C_{N}(G)}$, where $C_{S}(G)\cup C_{N}(G)$ is the graph obtained from $G$ by deleting the edges between
vertices in $\Core{G}$ and $\Npart{G}$. 
On the other hand, $\Null{C_{S}(G)\cup C_{N}(G)}$ is the orthogonal complement of $\Rank{C_{S}(G)\cup C_{N}(G)}=\up{\Rank{C_{S}(G)}}{$G$}{$C_{S}(G)$} \oplus \;\; \up{\Rank{C_{N}(G)}}{$G$}{$C_{N}(G)$}$. Therefore the result follows.
\item This result follows from the fact that the fundamental spaces of a graph are the direct sum of the fundamental
spaces of its connected components.
\end{enumerate}
\end{proof}

\section{On maximum independent sets}\label{sectionindep}
By $\mathcal{I}(G)$ we denote the set of all maximum independent sets of $G$.
The aim of this section is to prove that in $C_{4k}$-free bipartite graphs the null partition of the vertex set 
coincides with the maximum independent set
partition introduced by Zito. More specifically, $\Supp{G}=\cap_{I\in \mathcal{I}(G)}I$,
$\Core{G}=\cup_{I\in \mathcal{I}(G)}I^C $, i.e. is the set of vertices that are in no maximum independent sets, and $\Npart{G}$ is the set of vertices
that are in some, but not all, maximum independent sets. We begin by showing this on $C_{S}(G)$ and $C_{N}(G)$, and
use the fact that any independent set of $G$ is an independent set of a spanning subgraph to obtain the result for $G$.

Notice that by Lemma \ref{nmatchedton}, $C_N(G)$ has a perfect matching; thus, given a maximum independent set $I$, at 
most one vertex from each edge of the matching is in $I$. Hence, $\alpha(C_N(G))\leq |V(C_N(G))|/2=\npart{G}/2$. On the 
other hand, as $C_N(G)$ is bipartite, $\alpha(G)\geq |V(C_N(G))|/2$ as taking all vertices on one side of the 
bipartition yields an independent set. Therefore we have the following.
\begin{lemma}\label{npartindepence}
If $G$ is a $C_{4k}$-free bipartite graph, 
then $\alpha(C_N(G))=\npart{G}/2$. 
\end{lemma}
\begin{proof}
The lemma follows from the discussion preceding it.
\end{proof}

Using the fact that in $C_N(G)$ each set of the bipartition has exactly $\npart{G}/2$ vertices one can prove that all the 
vertices in $C_N(G)$ are in some but not all maximum independent sets of $C_N(G)$.
\begin{corollary}\label{zitoforcn}
If $G$ be a $C_{4k}$-free bipartite graph, and $v\in V(C_N(G))$, then there is a maximum independent set containing $v$
and a maximum independent set not containing $v$. 
\end{corollary}
\begin{proof}
Let $X,Y$ be the sets of the bipartition of $C_N(G)$, with $v\in X$. Then $X$ is a maximum independent set 
containing $v$, and $Y$ is a maximum independent set not containing $v$. 
\end{proof}

If $G$ is bipartite, by K\"{o}nig-Egerv\`{a}ry Theorem, $|V(G)|=\alpha(G)+\nu(G)$. Let $v\in \Supp{C_S(G)}$,
then by Theorem \ref{maintheorem} and Corollary \ref{matchingdecompo} there is a maximum matching of $C_S(G)$
such that $v\in U(M)$. Thus deleting $v$ does not change the matching number, $\nu(C_S(G)-v)=\nu(C_S(G))$.
Therefore $\alpha(C_S(G)-v)=\alpha(C_S(G))-1$, and $v$ is in every maximum matching of $C_S(G)$.
We have proved the following.
\begin{lemma}\label{suppinindependentset}
If $G$ is a $C_{4k}$-free bipartite graph, 
and $v\in \Supp{G}=\Supp{C_S(G)}$, then $v$ is in every maximum independent set of $C_S(G)$.
\end{lemma}
\begin{proof}
The lemma follows from the discussion preceding it.
\end{proof}

As $\Core{C_S(G)}=N(\Supp{C_S(G)})$, the previous lemma implies that the vertices in $\Core{C_S(G)}$ are in no maximum 
independent
set of $C_S(G)$. Thus $C_S(G)$ has a unique maximum independent set, i.e., $\Supp{C_S(G)}$.
\begin{corollary}\label{uniqueindepsetincs}
If $G$ is a $C_{4k}$-free bipartite graph, then $C_S(G)$ has a unique maximum independent set, i.e. $\Supp{C_S(G)}$.
\end{corollary}

Consider the graph $C_S(G)\cup C_N(G)$, obtained from $G$ by deleting all edges joining a vertex in $\Core{G}$ with
a vertex in $\Npart{G}$.
Combining Corollary \ref{zitoforcn} with Corollary \ref{uniqueindepsetincs} we get that any maximum independent set
of $C_S(G) \cup C_N(G)$ contains all vertices in $\Supp{C_S(G)}$ and no vertices in $\Core{C_S(G)}$. Notice that
it is also an independent set of $G$, as every deleted edge had at most one of its vertices in the independent set.
Therefore every maximum independent set of $G$ contains all vertices in $\Supp{C_S(G)}=\Supp{G}$, and for every
vertex in $\Npart{C_N(G)}=\Npart{G}$ there is a maximum independent set that contains it, and a maximum independent
set that does not contain it. Furthermore, $\Core{G}$ is the set of vertices that are not contained in any maximum
independent set. 
\begin{theorem}\label{vertchar}
Let $G$ be a $C_{4k}$-free bipartite graph. Then
\begin{enumerate}
\item $\Supp{G}$ is the set of vertices that are in every maximum independent set;
\item $\Core{G}$ is the set of vertices that are in none of the maximum independent sets;
\item $\Npart{G}$ is the set of vertices that are in some but not in all maximum independent sets.
\end{enumerate}
\end{theorem}
\begin{proof}
The theorem follows from the discussion preceding it.
\end{proof}
Theorem \ref{vertchar} says that null decomposition of \(C_{4k}\)-free bipartite graph coincides with the Zito decomposition, see \cite{zito1991structure}. Hence, null decomposition, Gallai-Edmonds decomposition, and Zito decomposition are equivalent in \(C_{4k}\)-free bipartite graphs. The graph \(G_{2}\) in Figure \ref{Fig_2} shows that these three decomposition are equivalent for some graphs that are not \(C_{4k}\)-free bipartite graphs.

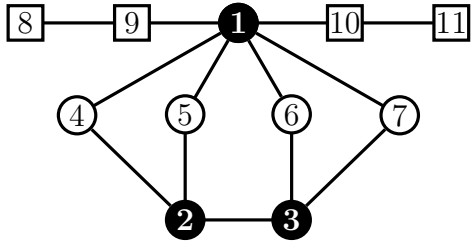
\begin{figure}[h]
	\centering
	\begin{tikzpicture}[scale=0.7]
	
	\draw[very thick] (0,0) node[n vertex] (8) 
	 {$8$}    		
	-- ++(0:2.0cm) node[n vertex] (9) {$9$}
	-- ++(0:2.0cm) node[c vertex] (1A) {}
	-- ++(0:2.0cm) node[n vertex] (10) {$10$}
	-- ++(0:2.0cm) node[n vertex] (11) {$11$};
	
	\draw[very thick] (1A)
	-- ++(210: 3.5cm) node[s vertex] (4) {$4$};
	
	\draw[very thick] (1A)
	-- ++(330: 3.5cm) node[s vertex] (7) {$7$};
	
	\draw[very thick] (1A) node[c vertex] {\textcolor{white}{$\mathbf{1}$}}
	-- ++(240: 2cm) node[s vertex] (5) {$5$}
	-- ++(270: 2cm) node[c vertex] (2) {\textcolor{white}{$\mathbf{2}$}}
	-- ++(0: 2cm) node[c vertex] (3) {\textcolor{white}{$\mathbf{3}$}}
	-- ++(90: 2cm) node[s vertex] (6) {$6$};
	
	\draw[very thick] (1A)--(6);
	\draw[very thick] (4)--(2);
	\draw[very thick] (7)--(3);
	\end{tikzpicture}
	\caption{Graph \(G_{2}\)} \label{Fig_2}
\end{figure}

But in general the null decomposition is not even a partition, and Gallai-Edmonds and Zito decomposition give different sets of vertices. For the graph \(G_{3}\) in Figure \ref{Fig_3} we have that:
\begin{enumerate}
	\item \(\Supp{G_{3}}=\{2,3,4,5,6,7,8\}\),
	\item \(\Core{G_{3}}=\{1,5,6,7,8\}\),
	\item \(\Npart{G_{3}}=\{9,10,11,12,13\}\),
	\item \(\mathcal{D}(G_{3})= \{2,3,4,9,10,11,12,13\}\)
	\item \(\mathcal{A}(G)=\{1\}\),
	\item \(\mathcal{C}(G)=\{5,6,7,8\}\).
	\item \(\bigcap_{I \in \mathcal{I}(G_3)} I=\{2,3,4\}\),
	\item \(\bigcap_{I \in \mathcal{I}(G_3)} I^{c}=\{1\}\),
\end{enumerate}

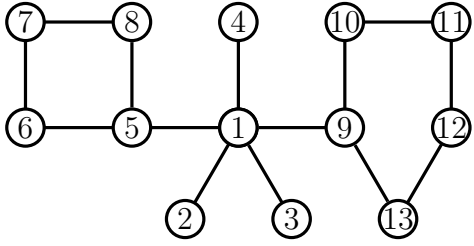
\begin{figure}[h]
	\centering
	\begin{tikzpicture}[scale=0.7]
	
	\draw[very thick] (0,0) node[s vertex] (2) 
	{$2$}    		
	-- ++(60:2.0cm) node[s vertex] (1A) {}
	-- ++(90:2.0cm) node[s vertex] (4) {$4$};
	
	\draw[very thick] (1A)
	-- ++(300:2.0cm) node[s vertex] (3) {$3$};
	
	\draw[very thick] (1A)
	-- ++(180:2.0cm) node[s vertex] (5) {$5$}
	-- ++(180:2.0cm) node[s vertex] (6) {$6$}
	-- ++(90:2.0cm)  node[s vertex] (7) {$7$}
	-- ++(0:2.0cm)   node[s vertex] (8) {$8$};
	
	\draw[very thick] (1A) node[s vertex] (1) {$1$}
	-- ++(0:2.0cm)   node[s vertex] (9) {$9$}
	-- ++(90:2.0cm)  node[s vertex] (10) {$10$}
	-- ++(0:2.0cm)   node[s vertex] (11) {$11$}
	-- ++(270:2.0cm) node[s vertex] (12) {$12$}
	-- ++(240:2.0cm) node[s vertex] (13) {$13$};

	\draw[very thick] (13)--(9);
	\draw[very thick] (8)--(5);
	\end{tikzpicture}
	\caption{Graph \(G_{3}\)} \label{Fig_3}
\end{figure}

In the proof of Theorem \ref{vertchar} we showed that any maximum independent set of $G$ is a union
of a maximum independent set of $\Npart{G}$ and a maximum independent set of $C_S{G}$. This result
is interesting on its own, and is stated as a corollary.
\begin{corollary}\label{indepdecompo}
Let $G$ be a $C_{4k}$-free bipartite graph.
If $I$ is a maximum independent set of $G$, then $I\cap \Npart{G}$ is a maximum independent set
of $C_N(G)$ and $I\cap \Supp{G}$ is a maximum independent set of $C_S(G)$.
\end{corollary}
Notice that Corollary \ref{indepdecompo} provides a parallel algorithm (on the number of components of \(C_{N}(G)\)) for listing and enumerating all the maximum independent set for any \(C_{4k}\)-free graph. This algorithm requires to know the null decomposition of the graph. 

As Corollary \ref{charactcn} is a characterization of $C_{4k}$-free bipartite graphs $G$ with $G=C_N(G)$, the
following corollary is a characterization for when $G=C_S(G)$.
\begin{corollary}
	Let \(G\) be a \(C_{4k}\)-free bipartite graph. \(G\) has a unique maximum independent set if and only if \(G=C_{S}(G)\).
\end{corollary}

As a final result for this section, we present the following formula
which follows from K\"{o}nig-Egerv\`{a}ry Theorem and Theorem \ref{corNullequality}.

\begin{corollary}\label{alfanulomatching}
If $G$ is a $C_{4k}$-free bipartite graph, then $\alpha(G)=\mnull{G}-\nu(G)$.
\end{corollary}
\begin{proof}
By K\"{o}nig-Egerv\`{a}ry Theorem, \(\nu(G)+\alpha(G)=|V(G)|\). Hence, by Theorem \ref{corNullequality}, \(\alpha(G)=\mnull{G}-\nu(G)\).
\end{proof}

\section{Conclusion}\label{sectionconclusion}
This paper studies the relations between
structural and spectral properties of a $C_{4k}$-free bipartite graph.
In particular the relation between the support of the null space of the adjacency
matrix, and the maximum matchings and maximum independent sets of the graph.
The following result summarizes many formulas that can be derived from these relations, 
and seems like a nice way to finish the paper. Here \(m(G)=|\mathcal{M}(G)|\) and \(a(G)=|\mathcal{I}(G)|\).
\begin{corollary}
Let $G$ be a $C_{4k}$-free bipartite graph. The following equalities hold.
\begin{enumerate}[(i)]
\item $\nu(G) =\core{G} + \frac{\npart{G}}{2}$,
\item $\rank{G}=2\core{G} + \npart{G}$,
\item $\mnull{G}=\supp{G}-\core{G}$,
\item $\alpha(G)=\supp{G}+\frac{\npart{G}}{2}$,
\item $m(G)=\prod_{S\in C_S(G)} m(S)\prod_{N\in C_N(G)} m(N)$,
\item $a(G)=\prod_{N\in C_N(G)}a(N)$.
\end{enumerate}
\end{corollary}
\begin{proof}
\begin{enumerate}[(i)]
\item Let $M$ be a maximum matching of $G$. By Corollaries \ref{corematchedtosupp} and \ref{nmatchedton},
every edge in $M$ has either both its vertices in $\Npart{G}$, or one vertex in $\Core{G}$ and one in $\Supp{G}$.
On the other hand every vertex in $\Npart{G}\cup\Core{G}$ is in one edge of $M$.
Thus $\nu(G)=|M| =\core{G} + \frac{\npart{G}}{2}$.
\item By Corollary \ref{rank2nu} $\rank{G}=2\nu(G)$. By item (i) we get $\rank{G}=2\core{G} + \npart{G}$.
\item The equality is obtained combining $|V(G)|=\supp{G}+\core{G}+\npart{G}$ with  $\mnull{G}=|V(G)|-\rank{G}$ and item (i).
\item Let $I$ be a maximum independent set of $G$. By Theorem \ref{vertchar} and Corollary \ref{indepdecompo}, $I=\left(I\cap \Supp{G}\right)\cup\left(I\cap \Npart{G}\right)$, and $I\cap \Supp{G}=\Supp{G}$. By Lemma \ref{npartindepence}, $|I\cap \Npart{G}|=\frac{\npart{G}}{2}$. Hence $\alpha(G)=|I|=\supp{G}+\frac{\npart{G}}{2}$.
\item By Corollary \ref{matchingdecompo} any maximum matching is obtained by the union of a maximum matching from $C_S(G)$ and a maximum matching from $C_N(G)$. Hence, the result follows.
\item By Corollary \ref{indepdecompo} any maximum independent set is obtained by the union of a maximum independent set from $C_S(G)$ and a maximum independent set from $C_N(G)$. But $C_S(G)$ has a unique maximum independent set, proving the result.
\end{enumerate}
\end{proof}


The inertia of a graph is the inertia of its adjacency matrix: \(\Inertia{G}=\Inertia{A(G)}\), with \(\Inertia{A(G)}=(a,b,c)\), where \(a\) is the number of negative eigenvalues of \(A(G)\), \(b\) is the multiplicity of zero as eigenvalue of \(A(G)\) and \(c\) is the number of positive eigenvalues of \(A(G)\).

\begin{theorem}\label{inertiath}
Let \(G\) be a \(C_{4k}\)-free bipartite graph. Then
\begin{enumerate}
	\item \(\Inertia{G}=\left(\core{G}+\dfrac{\npart{G}}{2},\;\;\supp{G}-\core{G},\;\;\core{G}+\dfrac{\npart{G}}{2}\right)\),
	\item \(\Inertia{G}= \sum \limits_{H \in \mathcal{K}(C_{S}(G))} \Inertia{H}+\sum\limits_{H \in \mathcal{K}(C_{N}(G))} \Inertia{H}\).
\end{enumerate}
\end{theorem}

As a corollary of Theorem \ref{inertiath}, in any $C_{4k}$-free bipartite graph the size of a maximum matching equals the number of positive eigenvalues. This result was previously known for trees and benzenoid graphs, see \cite{fajtlowicz2005maximum}.

\begin{corollary}
	For any
	\(C_{4k}\)-free bipartite graph, the size of a maximum matching equals the number of positive eigenvalues.
\end{corollary}

\section*{Acknowledgments}
The authors are gratefully indebted to Professor Mart\'{i}n Safe, Universidad Nacional del Sur, Bah\'{\i}a Blanca, Argentina, for many good ideas. With heartfelt, Daniel A. Jaume would like to thank Prof. Vilmar Trevisan for a wonderful work stay at Universidad Federal do Rio Grande do Sul, Porto Alegre, Brasil, and Prof. Glenn Hurlbert for a truly productive work stay at Virginia Commonwealth University, Richmond, United States.
\section*{}
This work was partially supported by the Universidad Nacional de San Luis, grant PROICO 03-0918, and MATH AmSud, grant 18-MATH-01. Dr. Daniel A. Jaume was partially funding by ``Programa de Becas de Integraci\'{o}n Regional para argentinos'', grant 2075/2017, and Fulbright Program.


\bibliographystyle{plain}

\bibliography{TAGcitas}

%
%
\end{document}